\documentclass{amsart}

\usepackage{graphicx}
\usepackage{amssymb}
\usepackage[all]{xy}
\usepackage{mathrsfs}
\usepackage{mathtools}
\usepackage{algorithm2e}
\RestyleAlgo{ruled}
\usepackage{tikz,tkz-euclide}
\usepackage[a4paper,margin=2.5cm]{geometry}
\usepackage{array}

\newcommand\undermat[2]{%
  \makebox[0pt][l]{$\smash{\underbrace{\phantom{%
    \begin{matrix}#2\end{matrix}}}_{\text{$#1$}}}$}#2}

\usepackage[breakable,skins]{tcolorbox}
\newtcolorbox{tbox}[1][]{%
    breakable,
    enhanced,
    colframe=blue,
    coltitle=white,
    #1
}

\usepackage{hyperref}
\usepackage{mathdots}
\usepackage[backrefs,lite]{amsrefs}

\hypersetup{
    colorlinks,
    citecolor=blue,
    filecolor=blue,
    linkcolor=blue,
    urlcolor=blue
}

\usepackage{tikz}
\usepackage{tikz-cd}
\usepackage{longtable}

\newtheorem{introthm}{Theorem}

\newtheorem{introcor}[introthm]{Corollary}

\newtheorem{theorem}{Theorem}[section]
\newtheorem{lemma}[theorem]{Lemma}
\newtheorem{proposition}[theorem]{Proposition}
\newtheorem{corollary}[theorem]{Corollary}

\theoremstyle{definition}

\newtheorem{example}[theorem]{Example}

\newtheorem{construction}[theorem]{Construction}

\newtheorem{remark}[theorem]{Remark}
\theoremstyle{remark}

\title{The Cox ring of an embedded variety}

\author[C.~Herrera]{Crist\'obal Herrera}
\address{
Departamento de Matem\'atica,
Universidad de Concepci\'on,
Casilla 160-C,
Concepci\'on, Chile}
\email{crherrera2019@udec.cl}

\author[A.~Laface]{Antonio Laface}
\address{
Departamento de Matem\'atica,
Universidad de Concepci\'on,
Casilla 160-C,
Concepci\'on, Chile}
\email{alaface@udec.cl}

\author[L.~Ugaglia]{Luca Ugaglia}
\address{
Dipartimento di Matematica e Informatica,
Universit\`a degli studi di Palermo,
Via Archirafi 34,
90123 Palermo, Italy}
\email{luca.ugaglia@unipa.it}

\subjclass[2010]{Primary 14M25; Secondary 14C20}
\keywords{Toric varieties, Mori dream spaces, Cox rings}
\thanks{The authors were partially supported by Proyecto FONDECYT Regular n.~1230287. The last two authors were also partially supported by the ``Piano strategico per il miglioramento della qualit\`a della ricerca e dei risultati della VQR 2020-2024 - Misura A'' of the University of Palermo. The third author is a member of INdAM - GNSAGA}

\date{\today}
\begin{document}

\begin{abstract}
We compute the Cox ring of an embedded variety $X \subseteq Z$ within a Mori dream space, under the assumption that the pullback map induces an isomorphism at the level of divisor class groups. We show that the Cox ring of $X$ is the intersection of finitely many localizations of a quotient image of the Cox ring of $Z$. As a consequence, we provide an algorithm that terminates if and only if the Cox ring of $X$ is finitely generated, thereby generalizing previous works on the subject~\cites{al,hau,ott,abban2}. We apply these results to compute the Cox ring of hypersurfaces in smooth projective toric varieties.
\end{abstract}

\maketitle

\section*{Introduction}
In what follows we work over an algebraically closed field $\mathbb K$ of characteristic $0$. Whenever we use the expression ``very general'', we assume in addition that $\mathbb K$ is uncountable; alternatively, the corresponding statements may be interpreted for the geometric generic member of the linear system.

The notion of the Cox sheaf and Cox ring, as treated in this paper, was first introduced by Hausen in~\cite{hau}, where he addressed for the first time the problem of computing the Cox ring of an embedded variety, relating it to the Cox ring of its ambient space.
In this work, we develop an algorithmic approach to this problem.
Our goal is to compute the Cox ring of a normal variety $X$ embedded in a normal variety $Z$ with finitely generated Cox ring $\mathcal R(Z)$.
Our strategy relies on a key assumption: the pullback map on divisor class groups
$\iota^\ast\colon {\rm Cl}(Z) \longrightarrow {\rm Cl}(X)$
is an isomorphism.
Recall that $Z$ comes with a natural quotient construction
\cite{adhl}*{Con. 1.6.3.1}:
\[
\begin{tikzcd}
\widehat Z \arrow[r, hook] \arrow[d, "p_Z"'] & \overline Z \\
Z &
\end{tikzcd}
\]
Here $\widehat Z$ is the relative spectrum of the Cox sheaf of $Z$, while
$\overline Z = \operatorname{Spec}(\mathcal R(Z))$ is the spectrum of the Cox ring.
The \emph{irrelevant ideal} $\mathscr I_{\mathrm{irr}}(Z) \subseteq \mathcal R(Z)$ is the ideal defining the closed subset $\overline Z \setminus \widehat Z$, which is called the \emph{irrelevant locus} of $Z$.
The morphism $p_Z$ is a good quotient for the action of the quasi-torus
$\operatorname{Spec}(\mathbb K[\operatorname{Cl}(Z)])$.
We denote by $Z_0 \subseteq Z$ the smooth locus of $Z$ and by $X_0 := X \cap Z_0$ its intersection with $X$.
We let $\widehat X$ be the Zariski closure of $p_Z^{-1}(X_0)$ in $\widehat Z$, and
$\tilde X$ its Zariski closure in $\overline Z$.
These spaces fit into the chain of inclusions
\[
  p_Z^{-1}(X_0)
  \subseteq
  \widehat X
  \subseteq
  \tilde X_1
  \subseteq
  \tilde X,
\]
where $\tilde X_1$ is the open subset of $\tilde X$ obtained by removing the codimension-one irreducible components of $\tilde X \setminus \widehat X$ (see Construction~\ref{con:emb} for details).
Throughout the paper, we say that an open subset $U \subseteq Y$ is \emph{big} in $Y$ if its complement $Y \setminus U$ has codimension at least $2$.

\begin{introthm}\label{mteo}
Let $Z$ be a normal variety with finitely generated Cox ring $\mathcal R(Z)$, and let $X\subseteq Z$ be a normal closed subvariety with $\Gamma(X,\mathcal O_X^*)=\mathbb K^*$. Assume that $X_0=X\cap Z_0$ is nonempty.
Assume that:
\begin{enumerate}
\item $p_Z^{-1}(X_0)$ is big in $\widehat X$;
\item $\tilde X_1$ is normal;
\item the pullback $\iota^\ast:{\rm Cl}(Z)\xrightarrow{\sim}{\rm Cl}(X)$ is an isomorphism.
\end{enumerate}
Set $R:=\mathcal R(Z)/I(\tilde X)$.
Let $m_1,\dots,m_s\in R$ be nonzero homogeneous elements which generate the ideal
$I(\tilde X\setminus\tilde X_1)$.
Then
\[
\mathcal R(X)= R_{m_1}\cap\cdots\cap R_{m_s},
\]
where the intersection is taken inside the common fraction field. If $\tilde X_1=\tilde X$, take $s=1$ and $m_1=1$.
\end{introthm}

Observe that when $Z$ is smooth and $\tilde X$ is normal, the hypotheses $(1)$ and $(2)$ are automatically satisfied.
Theorem~\ref{mteo} generalizes results in~\cite{hau}*{\S 2} (see also~\cite{adhl}*{Cor. 4.1.1.3}) and, in the case of hypersurfaces, the main result of~\cite{al}. Indeed, in these cases, one has $\tilde X_1 = \tilde X$. As a consequence of this description, we produce Algorithm~\ref{alg}, which computes the Cox ring of $X$ iteratively by adding new generators at each step, and which terminates if and only if the Cox ring is finitely generated (see Proposition~\ref{prop:termination}). Our first application is the following.

\begin{introthm}\label{main-teo}
Let $(\mathbb{P},H)$ be a polarized pair, where $\mathbb{P}$ is a normal projective toric variety of dimension $n \geq 3$ and $H$ is an ample Cartier divisor.
Consider the Cox ring
$\mathcal{R}(\mathbb{P}) = \mathbb{K}[T_1, \dots, T_r]$,
and let $X \in |H|$ be a hypersurface defined by a homogeneous polynomial $f \in \mathcal{R}(\mathbb{P})$.
If $n \geq 4$, we assume that $X$ is a general element of the linear system $|H|$.
If $n = 3$, we assume instead that $X$ is a very general element of $|H|$ and that $H+K_{\mathbb P}$ is a base point free Cartier divisor.

Assume that the irrelevant locus of $\mathbb P$ has a unique codimension--$2$ component, say $V(T_1,T_2)$.
Let $d>0$ be the multiplicity of $f$ along this component, so that $f$ admits a decomposition
\[
f = \sum_{i=0}^d (-1)^i f_i\, T_1^{d-i}\, T_2^i,
\qquad
\text{with } f_0,\dots,f_d \in \mathbb{K}[T_1, \dots, T_r] \text{ homogeneous}.
\]
Suppose further that the algebraic subset
$V(T_1, T_2, f_0, \dots, f_d) \subseteq \mathbb{K}^r$
has codimension $d+3$. Then the Cox ring of $X$ is isomorphic to
\[
\mathcal{R}(X) \simeq
\frac{\mathbb{K}[T_1, \dots, T_r, S_1, \dots, S_d]}{
\left\langle
\begin{array}{l}
f_0 + T_2 S_1, \\[3pt]
f_i + T_1 S_i + T_2 S_{i+1}, \quad 1 \leq i \leq d-1, \\[3pt]
f_d + T_1 S_d
\end{array}
\right\rangle},
\]
where $S_1,\dots,S_d$ are new homogeneous variables with
\[
\deg(S_i)=[H]-(d-i+1)\deg(T_1)-i\deg(T_2),\qquad 1\leq i\leq d.
\]
\end{introthm}

The above theorem generalizes results of Ottem \cite{ott}, who studied hypersurfaces in products of projective spaces, and recent findings by Denisi \cite{denisi} on hypersurfaces in products of weighted projective spaces. For this result
our technique resembles the unprojection methods developed by Miles Reid \cite{reid} and aligns with the approach taken by Abban in \cite{abban2}*{Sec. 3.3}.

In our next result we apply Theorem~\ref{main-teo} to the special case where $X$ is a general ample hypersurface in a smooth projective toric variety of Picard rank $2$.

\begin{introcor}
\label{main-cor}
Let $\mathbb{P}$ be a smooth projective
toric variety of dimension $n\geq 4$
and Picard rank $2$, and let $X$
be a smooth general ample hypersurface
of $\mathbb P$.
Then the Cox ring of $X$ is as given in Theorem~\ref{main-teo} in the following cases:
\begin{enumerate}
\item
The grading matrix and the irrelevant ideal of $\mathbb{P}$ are
\[
 \begin{bmatrix}
  1 & 1 & 0 & -a_1 & \cdots & -a_{n-1}\\
  0 & 0 & 1 & 1 & \cdots & 1
 \end{bmatrix},
 \qquad
 \langle T_1,T_2\rangle\cap \langle T_3,\dots,T_{n+2}\rangle,
\]
where $0\leq a_1\leq\cdots\leq a_{n-1}$
and $X$ has degree $[a,b]$, with $0<a\leq \#\{i\in\{1,\dots,n-1\}:a_i=0\}$ and $b > 0$.
In this case $d = a$.
\item
The grading matrix and the irrelevant ideal of $\mathbb{P}$ are
\[
 \begin{bmatrix}
  1 & \cdots & 1 & 0 & -a_1\\
  0 & \cdots & 0 & 1 & 1
 \end{bmatrix},
 \qquad
 \langle T_1,\dots,T_n\rangle\cap \langle T_{n+1},T_{n+2}\rangle,
\]
where $a_1\geq 0$
and $X$ has degree $[a,b]$, with $a>0$ and $0<b\leq n-1$. In this case $d = b$.
\end{enumerate}
\end{introcor}

This criterion allows us to compute the Cox rings of general anticanonical hypersurfaces in smooth toric Fano varieties of Picard rank two (see Corollary~\ref{mt}). We also examine ambient toric varieties of Picard rank three to provide examples of hypersurfaces whose Cox rings are finitely generated but not complete intersections.

The paper is organized as follows: In Section~\ref{sec:cox-ring-embedded}, we present the computation of the Cox ring for embedded varieties within Mori dream spaces, proving our main technical result, Theorem~\ref{mteo}. In Section~\ref{sec:algorithm}, we introduce Algorithm~\ref{alg}. Sections~\ref{sec:proof1} and~\ref{sec:proof2} provide the proofs of Theorem~\ref{main-teo} and Corollary~\ref{main-cor}, respectively. In Section~\ref{sec:applications}, we apply these results to the case where the ambient space is a smooth projective toric variety of Picard rank two, focusing first on general Calabi-Yau hypersurfaces and then on general ample surfaces in toric threefolds. Finally, in Section~\ref{sec:rank3}, we discuss two examples where the ambient space is a toric variety of Picard rank three.

\vspace{3mm}

{\bf Acknowledgements.}
It is a pleasure to thank Hamid Abban, Jarek Buczy\'nski, J\"urgen Hausen, and Bal\'azs Szendr\H{o}i for interesting discussions on the topic of this article.
We also thank the anonymous referee for the helpful comments and suggestions, which greatly contributed to improving the clarity and presentation of the paper.

\section{Cox rings and localizations}
\label{sec:cox-ring-embedded}

\subsection{Cox Sheaves}
We begin by recalling the definition of a Cox sheaf.
Let $\mathbb{K}$ be an algebraically closed field of characteristic $0$.
Let $X$ be a normal variety such that $\Gamma(X,\mathcal{O}^*) = \mathbb{K}^*$.
Denote by ${\rm WDiv}(X)$ the group of Weil divisors of $X$ and by ${\rm PDiv}(X)$ the subgroup of principal divisors.
The {\em divisor class group} of $X$ is defined as
\[
 {\rm Cl}(X) := \frac{{\rm WDiv}(X)}{{\rm PDiv}(X)}.
\]
To any subgroup $K \subseteq {\rm WDiv}(X)$, one can associate the following {\em sheaf of divisorial algebras}:
\[
 \mathcal{S} := \bigoplus_{D \in K} \mathcal{O}_X(D).
\]
Now, assume that the quotient map $K \to {\rm Cl}(X)$ is surjective, and let $K^0 \subseteq {\rm PDiv}(X)$ be its kernel.
Fix a homomorphism $\chi \colon K^0 \to \mathbb{K}(X)^*$ such that ${\rm div}(\chi(D)) = D$ for any $D \in K^0$.
Let $\mathcal{I} \subseteq \mathcal{S}$ be the ideal sheaf locally generated by elements of the form $1 - \chi(D)$, where $D$ ranges over $K^0$.
From now on, assume that ${\rm Cl}(X)$ is finitely generated.
The {\em Cox sheaf} is the ${\rm Cl}(X)$-graded sheaf of algebras defined as
\[
 \mathcal{R}_X := \mathcal{S} / \mathcal{I}.
\]
Choosing a different $K$ or a different $\chi$ does not change the isomorphism class of the Cox sheaf.
If $X$ is $\mathbb{Q}$-factorial, or if its Cox ring is finitely generated, then the Cox sheaf is locally of finite type~\cite{adhl}*{\S 1.6.1}.
In this case, we denote its relative spectrum by
\[
 \widehat{X} := {\rm Spec}_X \mathcal{R}_X.
\]
It can be shown that $\widehat{X}$ is an irreducible, normal, quasiaffine variety~\cite{adhl}*{Construction 1.6.1.3}.

\subsection{Cox Rings}
Under the same assumptions as in the previous subsection, the {\em Cox ring} of $X$ is defined as the ring of global sections of the Cox sheaf:
\[
 \mathcal{R}(X) := \Gamma(X, \mathcal{R}_X).
\]
It can be shown that $X$ can be covered by open subsets of the form
$X_{[D], f} := X \setminus {\rm Supp}({\rm div}(\tilde{f}) + D)$,
where $D$ is any representative for $[D]$, and $\tilde{f} \in \mathcal{O}_X(D)$ is a representative for $f$
(the open subset does not depend on the choice of $D$). The map
\[
 \Gamma(X, \mathcal{R}_X)_f \to \Gamma(X_{[D], f}, \mathcal{R}_X),
 \qquad \frac{g}{f^n} \mapsto \frac{g}{f^n}
\]
is well-defined and can be shown to be an isomorphism of graded algebras.
As a result, we obtain an embedding $\widehat{X} \to \overline{X}$, whose complement is a subset of codimension at least two, called the
{\em irrelevant locus} of $X$~\cite{adhl}*{Construction 1.6.3.1}.\footnote{The assumption in Construction 1.6.3.1 concerning the finite generation of the Cox ring is used only to ensure that $\overline{X}$ is a variety.}

The inclusion $\mathcal{O}_X \to \mathcal{R}_X$, as the degree-zero part of the Cox sheaf, induces a morphism $p_X\colon \widehat{X} \to X$, which is a good quotient for the action of the quasitorus $H_X := {\rm Spec}\, \mathbb{K}[{\rm Cl}(X)]$.
This situation is summarized in the following diagram:
\[
\begin{tikzcd}
\widehat{X} \arrow[r, hookrightarrow] \arrow[d, "p_X"'] & \overline{X} \\
X
\end{tikzcd}
\]
The morphism $p_X$ is called the {\em characteristic space} of $X$, while $\overline{X}$ is referred to as its {\em total coordinate space}. The \emph{irrelevant locus}
$\overline X \setminus \widehat X$
is the vanishing set of a distinguished homogeneous ideal
\(\mathscr I_{\mathrm{irr}}(X) \subseteq \mathcal R(X)\),
called the \emph{irrelevant ideal} of $X$ (see \cite{adhl}*{Def.~1.6.3.2}).

\begin{remark}\label{rem:irrelevant-ideal}
Recall that whenever $X$ is a projective variety and $H$ is any ample divisor on $X$ we have
\[
\mathscr I_{\mathrm{irr}}(X) =
\sqrt{\big\langle \bigoplus_{n>0} \mathcal R(X)_{[nH]} \big\rangle},
\]
that is, the irrelevant ideal is the radical of the ideal generated by all homogeneous pieces of $\mathcal R(X)$
corresponding to positive multiples of $H$ (see \cite{adhl}*{Def.~1.6.3.6}).
In particular, if $f \in \mathcal R(X)_{[nH]}$ for some $n>0$,
then $f$ lies in every prime ideal $\mathfrak p \subseteq \mathcal R(X)$ that contains
$\mathscr I_{\mathrm{irr}}(X)$.

When $X$ is a \emph{toric variety} with Cox ring
\(\mathcal R(X)=\mathbb K[T_1,\dots,T_r]\),
let $\Sigma$ be its defining fan and
$v_1,\dots,v_r$ the primitive generators of its one-dimensional cones.
The minimal prime ideals of the irrelevant ideal are precisely
\[
\langle T_i\mid i\in I\rangle,
\]
where $I$ is inclusion-minimal among the subsets of $\{1,\dots,r\}$
for which the rays $\{\mathbb Q_{\geq0}v_i:i\in I\}$ are not all contained
in any cone of $\Sigma$. For a simplicial fan these are the primitive
collections. For a nonsimplicial fan it is important to use containment
in a cone, rather than requiring the cone generated by these rays itself
to be a member of the fan; see \cite{cls}*{Prop.~5.1.6}.
\end{remark}

\subsection{Embedded Varieties}
We recall some basic facts about the Cox ring of an embedded variety~\cite{adhl}*{\S 4.1.1}.
Let $Z$ be a normal variety with a finitely generated Cox ring, and let $p_Z \colon \widehat{Z} \to Z$ be its characteristic space.
Denote by $Z_0$ the subset of smooth points of $Z$. Given a subvariety $X \subseteq Z$,
define $X_0 := X \cap Z_0$, let
$ \widehat{X}$ be the closure of
$p_Z^{-1}(X_0)$ in $\widehat Z$, and let
$p_X\colon \widehat{X}\to X$ denote the
restriction of $p_Z$. The
situation is summarized in the following
diagram
\[
\begin{tikzcd}
p_Z^{-1}(X_0) \arrow[d] \arrow[r, hook]
& \widehat{X} \arrow[d, "p_X"'] \arrow[r, hook]
& p_Z^{-1}(X) \arrow[d] \arrow[r, hook]
& \widehat{Z} \arrow[d, "p_Z"'] \\
X_0 \arrow[r, hook] & X \arrow[r, equals] & X \arrow[r, hook] & Z
\end{tikzcd}
\]

Whenever $X_0$ is big in $X$, we denote by
$K \subseteq {\rm WDiv}(Z)$ a subgroup which surjects onto ${\rm Cl}(Z)$ and such that $X$ is not contained in the support of any element of $K$.
Since Weil divisors on a normal variety are completely determined by their restriction to the smooth locus, every $D \in K$ admits a canonical pullback
\[
\iota^*(D) := \overline{D|_{X_0}} \in \mathrm{WDiv}(X),
\]
where the bar denotes Zariski closure in $X$.
Because $X_0$ is big in $X$ and $X$ is not contained in $\mathrm{Supp}(D)$, this defines a group homomorphism
\[
\xymatrix{
 K \ar[r]^-{\iota^*} &
 \operatorname{WDiv}(X),
}
\]
which maps principal divisors to principal divisors and, consequently, induces a pullback homomorphism
\[
\iota^* \colon \operatorname{Cl}(Z) \longrightarrow \operatorname{Cl}(X)
\]
on divisor class groups. The following extension of
\cite{adhl}*{Cor.~4.1.1.2} applies to a possibly singular ambient variety.
Its first hypothesis implies that $X_0$ is big in $X$, as shown in the proof,
so that the pullback used in its third hypothesis is defined.

\begin{proposition}
\label{pro:sheaf}
Let $Z$ be a normal variety with finitely generated Cox ring, and let
$p_Z:\widehat Z\to Z$ be a characteristic space for $Z$.
Let $X\subseteq Z$ be a normal closed subvariety with
$\Gamma(X,\mathcal O_X^*)=\mathbb K^*$.
Let $Z_0$ be the smooth locus of $Z$, and assume that
$X_0:=X\cap Z_0$ is nonempty.
Let $\widehat X$ be the closure of $p_Z^{-1}(X_0)$ in $\widehat Z$,
and let $p_X:\widehat X\to X$ be the restriction of $p_Z$.

Assume that:
\begin{enumerate}
\item $p_Z^{-1}(X_0)$ is big in $\widehat X$;
\item $\widehat X$ is normal;
\item the pullback
$\iota^*:\operatorname{Cl}(Z)\to\operatorname{Cl}(X)$
is an isomorphism.
\end{enumerate}
Then, for any Cox sheaf $\mathcal R_X$ on $X$, there is an
isomorphism of $\operatorname{Cl}(X)$-graded
$\mathcal O_X$-algebras
\[
\mathcal R_X\simeq(p_X)_*\mathcal O_{\widehat X},
\]
where the grading on the right-hand side is identified with
$\operatorname{Cl}(X)$ via $\iota^*$.
In particular, $p_X:\widehat X\to X$ is a characteristic space
for $X$.
\end{proposition}

\begin{proof}
Set $U:=p_Z^{-1}(X_0)$. Since a good quotient maps invariant
closed subsets to closed subsets, $p_Z(\widehat X)$ is closed and contains
the dense open subset $X_0$ of $X$. Hence $p_X$ is surjective.
Moreover, $X_0$ is big in $X$: otherwise, a local equation
of a prime divisor in $X\setminus X_0$, near its generic point,
would pull back to a function with a divisorial zero locus
in $\widehat X\setminus U$, contradicting the first assumption.

Let $j:X_0\hookrightarrow X$ and $k:U\hookrightarrow\widehat X$
be the inclusions, and write $q:=p_X|_U$, so that
$p_X\circ k=j\circ q$.
Restriction to big open subsets identifies the divisor class
groups, and the restriction of $p_Z$ over $Z_0$ is a characteristic
space for $Z_0$. Thus \cite{adhl}*{Cor.~4.1.1.2}, applied to
$X_0\subseteq Z_0$, gives
\[
q_*\mathcal O_U\simeq\mathcal R_X|_{X_0}.
\]
Normality of $\widehat X$ and bigness of $U$ yield
$\mathcal O_{\widehat X}\simeq k_*\mathcal O_U$.
Likewise, since the homogeneous components of $\mathcal R_X$
are reflexive and $X_0$ is big in $X$, we have
$\mathcal R_X\simeq j_*(\mathcal R_X|_{X_0})$.
Thus
\[
(p_X)_*\mathcal O_{\widehat X}
\simeq (p_X)_*k_*\mathcal O_U
= j_*q_*\mathcal O_U
\simeq j_*(\mathcal R_X|_{X_0})
\simeq\mathcal R_X
\]
as graded $\mathcal O_X$-algebras.
Finally, $p_X$ is affine, being the restriction of $p_Z$ to the closed subvariety $\widehat X$, hence
$\widehat X\simeq\operatorname{Spec}_X\mathcal R_X$.
\end{proof}

\begin{construction}
\label{con:emb}
Let $Z$ be a variety with a finitely generated Cox ring, and let $X \subseteq Z$ be a subvariety satisfying the hypotheses of Proposition~\ref{pro:sheaf}. Let $\tilde{X} \subseteq \overline{Z}$ be the closure of $\widehat{X}$. We define $\tilde{X}_1 \subseteq \tilde{X}$ as the open subset obtained by removing the union of all codimension-one irreducible components of $\tilde{X} \setminus \widehat{X}$. The following diagram represents this setup:
\[
\begin{tikzcd}
\widehat{X} \arrow[r, hook] \arrow[d, hook]
& \tilde{X}_1 \arrow[r, hook]
& \tilde{X} \arrow[d, hook] \\
\widehat{Z} \arrow[rr, hook] &&
\overline{Z}
\end{tikzcd}
\]
In this diagram $\widehat X$ is big in $\tilde{X}_1$, while the complement $\tilde{X} \setminus \tilde{X}_1$ is a union of hypersurfaces.
\end{construction}

\begin{proof}[Proof of Theorem~\ref{mteo}]
Since $\tilde X_1$ is normal and $\widehat X$ is an open subset of $\tilde X_1$, the variety $\widehat X$ is normal. Therefore Proposition~\ref{pro:sheaf} applies. \begin{align*}
 \mathcal{R}(X)
 &= \Gamma(X, \mathcal{R}_X) \\[5pt]
 &= \Gamma(X, (p_X)_*\mathcal{O}_{\widehat{X}})
 & \text{by Proposition~\ref{pro:sheaf}} \\[5pt]
 &= \Gamma(\widehat{X}, \mathcal{O}) \\[5pt]
 &= \Gamma(\tilde{X}_1, \mathcal{O})
 & \text{since $\widehat X$ is big in $\tilde{X}_1$ and the latter is normal} \\[5pt]
 &= \Gamma(\tilde{X} \setminus V(m_1, \dots, m_s), \mathcal{O}) \\[5pt] 
 &= \Gamma\Big(\bigcup_{i=1}^s \tilde{X}_{m_i}, \mathcal{O}\Big) \\[5pt]
 &= \bigcap_{i=1}^s \Gamma(\tilde{X}_{m_i}, \mathcal{O}) \\[5pt]
 &= \bigcap_{i=1}^s (\mathcal{R}(Z)/I(\tilde{X}))_{m_i}
 & \text{since } \Gamma(\tilde{X}, \mathcal{O}) \simeq \mathcal{R}(Z)/I(\tilde{X}).
\end{align*}

\end{proof}

\begin{remark}
\label{rem:coarse-version}
There is also a coarser version of Theorem~\ref{mteo}. Let
$m_1,\dots,m_s\in R$ be the images of homogeneous generators of the irrelevant ideal
$\mathscr I_{\mathrm{irr}}(Z)\subseteq\mathcal R(Z)$. Since
\[
 \widehat X=\tilde X\cap\widehat Z
 =
 \tilde X\setminus V(m_1,\dots,m_s),
\]
the preceding proof gives directly
\[
 \mathcal R(X)=\Gamma(\widehat X,\mathcal O_{\widehat X})
 =
 \bigcap_{i=1}^s R_{m_i}.
\]
This formulation uses the whole irrelevant ideal of $Z$ and does not require the intermediate open subset $\tilde X_1$. The formulation in Theorem~\ref{mteo} is sharper and is the one used in the applications below, because it only localizes along the divisorial part of $\tilde X\setminus\widehat X$.
\end{remark}

\begin{remark}
\label{rem:normal}
In Theorem~\ref{mteo}, if the coordinate ring of $\tilde{X} \subseteq \overline{Z}$ is Cohen-Macaulay, and in particular if it is a complete intersection, then Serre's condition $S_2$ is automatically satisfied~\cite{stacks-project}*{\href{https://stacks.math.columbia.edu/tag/0342}{Tag 0342}}. Since this condition is local~\cite{stacks-project}*{\href{https://stacks.math.columbia.edu/tag/033Q}{Tag 033Q}}, we conclude that in this case, $\tilde{X}_1$ is normal if and only if it is smooth in codimension one.
\end{remark}

\section{Main algorithm}
\label{sec:algorithm}

In this section, we present an algorithm for computing the intersection of
finitely many localizations of a normal domain. Throughout the section, $R$
is a finitely generated normal domain over a field, with fraction field
$K=\operatorname{Frac}(R)$, and $m_1,\dots,m_s\in R$ are non-zero elements.
All intersections of localizations are taken inside $K$.
To apply the algorithm to Theorem~\ref{mteo}, first replace its quotient
ring $R$ by its normalization $R^\nu$ if necessary. Indeed, each
$R_{m_i}$ is normal because $D(m_i)\subseteq\tilde X_1$ is normal, and
normalization commutes with localization, so
$(R^\nu)_{m_i}=R_{m_i}$. Thus this preliminary step does not change the
intersection to be computed.

For any intermediate ring $A$ such that
$R\subseteq A\subseteq R_{m_1}\cap\cdots\cap R_{m_s}$, and for every
integer $k>0$, we set
\[
 I_k(A):=
 \bigcap_{i=2}^s
 \langle m_1^k\rangle_A:\langle m_i\rangle_A^\infty .
\]
Thus $I_k(A)$ consists of those elements $f\in A$ such that
$f/m_1^k$ belongs to $A_{m_i}$ for every $i=2,\dots,s$.

\begin{algorithm}[H]
\label{alg}
\caption{Computing the intersection of localizations}
    \SetKwInOut{Input}{Input}
    \SetKwInOut{Output}{Output}
    \Input{A finitely generated normal domain $R$ over a field and non-zero elements $m_1,\dots,m_s\in R$}
    \Output{$R_{m_1}\cap\cdots \cap R_{m_s}$, if the algorithm terminates}

    \While{$V(m_1,\dots,m_s)$ has codimension $1$ in ${\rm Spec}(R)$}
    {
        Search the smallest integer $n>0$ such that
        $I_n(R)\not\subseteq \langle m_1^n\rangle_R$.

        Choose a finite generating set of $I_n(R)$, and let $B_n$ be the
        subset of its elements which do not belong to $\langle m_1^n\rangle_R$.

        Define
        \[
        R_0 :=
        R[S_f\mid f\in B_n]\Big/
        \left(
        \langle S_fm_1^n-f\mid f\in B_n\rangle
        :\langle m_1\rangle^\infty
        \right).
        \]

        Replace $R$ by the normalization of $R_0$ in $K$.
    }

    \Return $R$\;
\end{algorithm}

\medskip

This algorithm has been implemented in Magma~\cite{magma} and can be freely
downloaded from:
\begin{center}
 \url{https://github.com/alaface/Cox-subvarieties.git}
\end{center}

\begin{remark}
When $R$ is a graded ring, as in the case of Cox rings, and the elements
$m_1,\dots,m_s$ are homogeneous, the ideals $\langle m_1^k\rangle_R$ and
$\langle m_i\rangle_R$ are homogeneous. Since colon ideals and saturations of
homogeneous ideals are homogeneous, each ideal $I_k(R)$ is homogeneous.

Therefore one may choose a homogeneous generating set for $I_k(R)$. With this
choice, the elements of $B_n$ are homogeneous. If one gives $S_f$ the degree
$\deg(f)-n\deg(m_1)$, then the relation $S_fm_1^n-f$ is homogeneous.
Thus the quotient appearing in Algorithm~\ref{alg} is graded. Moreover, the
normalization of a finitely generated graded domain inside its homogeneous
fraction field is again graded. Hence the algorithm is compatible with the Cox
ring grading.
\end{remark}

Before proving the algorithm, we recall two elementary lemmas.

\begin{lemma}
\label{pro:crit}
Let $A$ be a noetherian normal domain, and let $m_1,\dots,m_s\in A$ be
non-zero elements. Then
\[
A=\bigcap_{i=1}^s A_{m_i}
\quad\Longleftrightarrow\quad
\operatorname{codim}_{\operatorname{Spec}(A)}V(m_1,\dots,m_s)\ge 2.
\]
Equivalently, if $R\subseteq A\subseteq R_{m_1}\cap\cdots\cap R_{m_s}$ and
$A$ is a noetherian normal domain, then $A=R_{m_1}\cap\cdots\cap R_{m_s}$
if and only if $V(m_1,\dots,m_s)$ has codimension at least $2$ in
${\rm Spec}(A)$.
\end{lemma}

\begin{proof}
Let $K=\operatorname{Frac}(A)$. Since $A$ is a noetherian normal domain, it
is a Krull domain. Hence $A$ is the intersection, inside $K$, of the
local rings $A_{\mathfrak p}$, where $\mathfrak p$ runs over the
height-one primes of $A$. Similarly, $A_m$ is obtained by intersecting only
those $A_{\mathfrak p}$ such that $m\notin\mathfrak p$.

It follows that $\bigcap_i A_{m_i}$ is the intersection of the height-one
localizations $A_{\mathfrak p}$ such that $\mathfrak p$ does not contain
all the elements $m_1,\dots,m_s$. If $V(m_1,\dots,m_s)$ has codimension at
least $2$, no height-one prime contains all the $m_i$. Therefore the above
intersection is the intersection over all height-one primes, hence it is $A$.

Conversely, suppose that $V(m_1,\dots,m_s)$ has codimension $1$. Then there
is a height-one prime $\mathfrak p$ containing $(m_1,\dots,m_s)$. Since
$A$ is a Krull domain, the approximation theorem gives an element
$x\in K$ with negative valuation along $\mathfrak p$ and non-negative
valuation along every other height-one prime. Then $x\notin A$. However,
after localizing at any $m_i$, the prime $\mathfrak p$ disappears, because
$m_i\in\mathfrak p$. Hence $x\in A_{m_i}$ for every $i$. Therefore
$A\subsetneq \bigcap_i A_{m_i}$.

The final statement follows because, if
$R\subseteq A\subseteq R_{m_1}\cap\cdots\cap R_{m_s}$, then localizing at
$m_i$ gives $A_{m_i}=R_{m_i}$ for every $i$.
\end{proof}

\begin{lemma}
\label{lem:sat}
Let $A$ be an integral domain, let $m\in A$ be non-zero, and let
$f_1,\dots,f_t\in A$, $r_1,\dots,r_t>0$. Let
$A'=A[f_1/m^{r_1},\dots,f_t/m^{r_t}]\subseteq A_m$. Then
\[
A'\simeq
A[S_1,\dots,S_t]\Big/
\left(
\langle S_1m^{r_1}-f_1,\dots,S_tm^{r_t}-f_t\rangle
:\langle m\rangle^\infty
\right).
\]
\end{lemma}

\begin{proof}
Consider the natural surjective homomorphism
$\varphi:A[S_1,\dots,S_t]\to A'$ sending $S_j$ to $f_j/m^{r_j}$.
After localizing at $m$, its kernel is generated by
$S_j-f_j/m^{r_j}$, or equivalently by $S_jm^{r_j}-f_j$, for
$j=1,\dots,t$. The kernel of $\varphi$ is the contraction of this localized
kernel to $A[S_1,\dots,S_t]$. This contraction is precisely the saturation
with respect to $\langle m\rangle$. This proves the claim.
\end{proof}

\begin{proof}[Proof of Algorithm~\ref{alg}]
Let $R_0$ denote the initial ring, and set
$T:= (R_0)_{m_1}\cap\cdots\cap (R_0)_{m_s}\subseteq K$.
We prove by induction that the current ring $A$ is a finitely generated
normal domain contained in $T$, and that
$A_{m_i}=(R_0)_{m_i}$ for every $i$. This is clear at the beginning,
where $A=R_0$.

Assume that these properties hold for the current ring $A$. If
$V(m_1,\dots,m_s)$ has codimension at least $2$ in ${\rm Spec}(A)$, then
Lemma~\ref{pro:crit} gives $A=T$, and the algorithm returns the desired
intersection.

Suppose instead that $V(m_1,\dots,m_s)$ has codimension $1$ in
${\rm Spec}(A)$. Again by Lemma~\ref{pro:crit}, we have $A\neq T$.
Choose $x\in T\setminus A$. Since $T\subseteq A_{m_1}$, we may write
$x=f/m_1^n$, with $f\in A$ and $n>0$. Since $x\notin A$, one has
$f\notin\langle m_1^n\rangle_A$. Moreover, the condition
$x\in A_{m_i}$, for $i\geq 2$, is equivalent to saying that some power of
$m_i$ sends $f$ into $\langle m_1^n\rangle_A$. In other words,
$f\in I_n(A)$. Thus $I_n(A)\not\subseteq \langle m_1^n\rangle_A$, so the
integer searched for by the algorithm exists.

Let $B_n$ be obtained from a finite generating set of $I_n(A)$ by keeping
only those generators which do not belong to $\langle m_1^n\rangle_A$.
Adjoining the fractions $f/m_1^n$, with $f\in B_n$, is enough to adjoin
$h/m_1^n$ for every $h\in I_n(A)$, since the discarded generators give
fractions already contained in $A$.

By Lemma~\ref{lem:sat}, the saturated quotient appearing in the algorithm is
the subring of $A_{m_1}$ generated by $A$ and the fractions
$f/m_1^n$, with $f\in B_n$. All these fractions belong to $T$, so this
new ring is contained in $T$. Since $T$ is an intersection of normal
domains inside $K$, it is normal. Therefore the normalization of the new
ring in $K$ is still contained in $T$.

Finally, localizing the new ring at any $m_i$ gives the same localization as
before: the new ring contains $A$ and is contained in $A_{m_i}$. Thus
after localization at $m_i$ we get
\[
A_{m_i}\subseteq A'_{m_i}\subseteq A_{m_i},
\]
where $A'$ denotes the normalized new ring. Hence
$A'_{m_i}=A_{m_i}=(R_0)_{m_i}$ for every $i$.
The inductive properties are preserved, and this proves the correctness of
the algorithm.
\end{proof}

\begin{proposition}
\label{prop:termination}
Algorithm~\ref{alg} terminates if and only if
$R_{m_1}\cap\cdots\cap R_{m_s}$ is a finitely generated $R$-algebra.
\end{proposition}

\begin{proof}
Let $T:=R_{m_1}\cap\cdots\cap R_{m_s}$.

If the algorithm terminates, then by the proof of correctness the final ring is
$T$. At each step one adjoins finitely many fractions and then takes the
normalization of a finitely generated domain over a field. Since normalization
is finite in this setting, the final ring is a finitely generated
$R$-algebra. Thus $T$ is finitely generated over $R$.

Conversely, assume that $T$ is finitely generated over $R$. Choose
generators $\theta_1,\dots,\theta_t$ such that $T=R[\theta_1,\dots,\theta_t]$.
Let $A$ be an intermediate ring appearing during the algorithm. Since
$R\subseteq A\subseteq T$, the same elements generate $T$ as an
$A$-algebra.
For each $j$, write $\theta_j=a_j/m_1^{r_j}$ with $a_j\in A$, choosing
$r_j$ minimal. Let $\rho(A)$ be the maximum of these exponents. If
$\rho(A)=0$, then all $\theta_j$ belong to $A$, hence $A=T$, and the algorithm
stops by Lemma~\ref{pro:crit}.

Assume that $\rho(A)>0$, and let $n$ be the integer chosen by the algorithm.
If $r_j>0$, then $\theta_j\notin A$, so $a_j\notin \langle m_1^{r_j}\rangle_A$.
Since $\theta_j\in T$, we have $a_j\in I_{r_j}(A)$; hence, by the minimality
of $n$, one has $n\leq r_j$.

Moreover $a_j/m_1^n$ belongs to $T$, and $a_j\in I_n(A)$. The new ring
contains all fractions $h/m_1^n$ with $h\in I_n(A)$, because these are
generated by the fractions associated with a set of generators of $I_n(A)$,
up to elements already in $A$. Therefore, in the next intermediate ring,
$\theta_j$ can be written with denominator exponent at most $r_j-n$.
Thus every positive denominator exponent strictly decreases, while
zero exponents remain zero. Hence $\rho(A)$ decreases by at least one at
each nonterminal step, proving termination.
\end{proof}

At each step of the algorithm, one has to compute rings obtained by adjoining
fractions to $R$, typically by means of saturation and normalization. The next
criterion provides a way to certify finite generation of the intersection of
localizations without performing all the intermediate saturation and
normalization computations. Namely, once a finite set of candidate generators
has been found, it suffices to verify a codimension condition and normality.

\begin{corollary}
\label{cor:r2}
Let $R$ be a normal finitely generated $\mathbb K$-domain and let
$m_1,\dots,m_k\in R$ be non-zero elements. Assume that
$q_j:=f_j/m_1^{r_j}$ belongs to $R_{m_1}\cap\cdots\cap R_{m_k}$ for
$j=1,\dots,t$. Let $P:=R[s_1,\dots,s_t]$ and let
\[
I:=\langle s_1m_1^{r_1}-f_1,\dots,s_tm_1^{r_t}-f_t\rangle.
\]
Assume that there exists an ideal $J\subseteq P$ such that
\[
I\subseteq J\subseteq I:\langle m_1\rangle^\infty
\]
and, setting $A:=P/J$, assume that $A$ is equidimensional and that
\[
\operatorname{codim}_{\operatorname{Spec}(A)}V(m_1,\dots,m_k)\geq 2.
\]
Let $\overline A$ be the image of $A$ in $R_{m_1}$ under the homomorphism
$s_j\mapsto q_j$. Equivalently, $\overline A\simeq P/(I:\langle m_1\rangle^\infty)$.
Let $\overline A^\nu$ be the normalization of $\overline A$ in
$\operatorname{Frac}(R)$. Then
\[
\overline A^\nu=R_{m_1}\cap\cdots\cap R_{m_k}.
\]
In particular, if $\overline A$ is normal, it is already the desired
intersection. Alternatively, equidimensionality of $A$ may be omitted
if the codimension hypothesis is imposed directly in
$\operatorname{Spec}(\overline A)$.
\end{corollary}
\begin{proof}
Consider the homomorphism $\varphi:P\to R_{m_1}$ defined by
$s_j\mapsto q_j$. By Lemma~\ref{lem:sat}, its kernel is
$\ker(\varphi)=I:\langle m_1\rangle^\infty$. Hence the image of $A=P/J$
under $\varphi$ is $\overline A\simeq P/(I:\langle m_1\rangle^\infty)$ and,
under this identification, $\overline A=R[q_1,\dots,q_t]\subseteq R_{m_1}$.
Since each $q_j$ belongs to $R_{m_1}\cap\cdots\cap R_{m_k}$, we have
$\overline A\subseteq R_{m_i}$ for every $i=1,\dots,k$.
Since $A_{m_1}=R_{m_1}$ is a domain, the kernel of
$A\to\overline A$ is the unique minimal prime of $A$ not containing
$m_1$. Thus $\operatorname{Spec}(\overline A)$ is an irreducible
component of $\operatorname{Spec}(A)$. By equidimensionality it has the
same dimension as $A$. The assumed codimension bound on $A$ therefore
implies the same bound on $\operatorname{Spec}(\overline A)$.
If this latter bound is assumed directly, no equidimensionality
assumption is needed.
For every $i$, localizing the inclusions
$R\subseteq\overline A\subseteq R_{m_i}$ at $m_i$ gives
$\overline A_{m_i}=R_{m_i}$. Since normalization commutes with localization
and $R_{m_i}$ is normal, we get $(\overline A^\nu)_{m_i}=R_{m_i}$ for all
$i=1,\dots,k$. Moreover, $\overline A^\nu$ is finite over $\overline A$,
because $\overline A$ is a finitely generated $\mathbb K$-algebra. Hence the
inverse image of $V(m_1,\dots,m_k)$ in $\operatorname{Spec}(\overline A^\nu)$
still has codimension at least $2$.
Now $\overline A^\nu$ is normal. Applying Lemma~\ref{pro:crit} to
$\overline A^\nu$, we obtain
\[
\overline A^\nu
=
(\overline A^\nu)_{m_1}\cap\cdots\cap(\overline A^\nu)_{m_k}
=
R_{m_1}\cap\cdots\cap R_{m_k}.
\]
If $\overline A$ is normal, its normalization can be omitted.
\end{proof}

\section{Proof of Theorem~\ref{main-teo}}
\label{sec:proof1}

\begin{proof}[Proof of Theorem~\ref{main-teo}]
We first establish the geometric properties needed for
Theorem~\ref{mteo}. The restriction map
\[
\operatorname{Cl}(\mathbb P)\xrightarrow{\sim}\operatorname{Cl}(X)
\]
is an isomorphism by \cite{rs}*{Thm.~1} when $n\geq4$.
When $n=3$, the Cartier assumption on $H+K_{\mathbb P}$ implies that
$\mathbb P$ is Gorenstein. By \cite{ko}*{Prop.~1}, ampleness of $H$
and global generation of $H+K_{\mathbb P}$ imply that $H$ is very ample.
Since toric varieties have rational singularities, the
Grauert--Riemenschneider canonical sheaf equals $\omega_{\mathbb P}$.
Thus \cite{rs2}*{Thm.~1} applies and gives the same isomorphism for
very general $X$.

By Bertini, $X$ is normal and $X_0=X\cap\mathbb P_{\mathrm{reg}}$
is a nonempty big open subset of $X$.
The characteristic space $\widehat{\mathbb P}$ is a smooth open subset
of $\mathbb A^r$, and the Cox-coordinate sections of $H$ have no common
zero on it. Bertini therefore shows that
$V(f)\cap\widehat{\mathbb P}$ is smooth for general $f$.
Moreover, $p_{\mathbb P}^{-1}(\mathbb P_{\mathrm{Sing}})$ has
codimension at least two in $\widehat{\mathbb P}$: points with at most
one zero Cox coordinate map to the torus or to a codimension-one torus
orbit, both contained in $\mathbb P_{\mathrm{reg}}$.
A general section contains none of the components of this inverse
image, so its intersection with $V(f)$ still has codimension at least
two in $V(f)\cap\widehat{\mathbb P}$.
Consequently
\[
\widehat X=V(f)\cap\widehat{\mathbb P},\qquad
p_{\mathbb P}^{-1}(X_0)\text{ is big in }\widehat X.
\]
The space $p_{\mathbb P}^{-1}(X_0)$ is a characteristic space for $X_0$
by \cite{adhl}*{Cor.~4.1.1.2}, hence is irreducible.
It follows that $\widehat X$ is irreducible. Since the irrelevant locus
has codimension at least two in $\mathbb A^r$, every irreducible
component of $V(f)$ meets $\widehat{\mathbb P}$. Smoothness there also
excludes repeated factors. Thus $f$ is irreducible,
$\tilde X=V(f)$, and
\[
A:=\mathbb K[T_1,\dots,T_r]/(f)
\]
is a domain.

Every ample homogeneous section vanishes on the irrelevant locus.
Its unique codimension-two component $V(T_1,T_2)$ is therefore a
divisor in $\tilde X$. All other irrelevant components have ambient
codimension at least three, hence codimension at least two in
$\tilde X$, even though they are also contained in $V(f)$. Therefore
\[
\tilde X_1=\tilde X\setminus V(T_1,T_2)=D(T_1)\cup D(T_2).
\]
This open subset is normal: it is a hypersurface, hence satisfies
$S_2$, and its complement of the smooth open subset $\widehat X$
has codimension at least two, so it satisfies $R_1$.
In particular, $A_{T_1}$ and $A_{T_2}$ are normal.

Set $P:=\mathbb K[T_1,\dots,T_r,S_1,\dots,S_d]$ and let $I$ be
generated by
\[
g_0=f_0+T_2S_1,\quad
g_i=f_i+T_1S_i+T_2S_{i+1}\ (1\leq i\leq d-1),\quad
g_d=f_d+T_1S_d.
\]
Write $C:=P/I$. Eliminating the $S_i$ recursively after inverting
$T_2$, or in the reverse order after inverting $T_1$, gives
\[
C_{T_2}\simeq A_{T_2},\qquad C_{T_1}\simeq A_{T_1}.
\]
Indeed, the remaining equation is
$\sum_{i=0}^d(-1)^if_iT_1^{d-i}T_2^i=f$.
These identifications agree after inverting $T_1T_2$.

The hypothesis on the coefficients gives
\[
\dim C/(T_1,T_2)=r-3.
\]
Since $I$ has $d+1$ generators, every minimal component of
$\operatorname{Spec}(C)$ has dimension at least $r-1$.
No such component can therefore be contained in $V(T_1,T_2)$.
The preceding localizations show that their union outside this locus
is irreducible, reduced and of dimension $r-1$.
Thus $I$ has height $d+1$, and $C$ is a complete intersection.
It is Cohen--Macaulay, with a unique minimal prime, and generically
reduced; hence it is a domain.

The locus $V(T_1,T_2)$ now has codimension two in
$\operatorname{Spec}(C)$. Its complement is normal by the localization
identifications, so $C$ satisfies $R_1$. As it also satisfies $S_2$,
Serre's criterion gives normality of $C$. Extension across a subset
of codimension two yields
\[
C=C_{T_1}\cap C_{T_2}=A_{T_1}\cap A_{T_2}.
\]
All the hypotheses of Theorem~\ref{mteo} have been verified, and that
theorem identifies the last intersection with $\mathcal R(X)$.
\end{proof}

\begin{remark}
\label{main-rem}
Let $(\mathbb{P},H)$ be a polarized pair consisting of a
normal projective toric variety of dimension at least $4$
and an ample Cartier divisor.
Let $X\in |H|$ be general, defined by a homogeneous polynomial $f \in \mathcal{R}(\mathbb{P}) = \mathbb{K}[T_1, \dots, T_r]$ in Cox coordinates.
Let $R := \mathbb{K}[T_1, \dots, T_r]/\langle f\rangle$.
Assume that the irrelevant locus of $\mathbb P$ has a component of codimension $2$, say $V(T_1,T_2)$, and that every monomial generator of the irrelevant ideal is divisible by either $T_1$ or $T_2$.
Let $d>0$ be the multiplicity of $f$ along this component. We can write
\[
 f = \sum_{i=0}^d(-1)^if_iT_1^{d-i}T_2^i.
\]
Then the Cox ring of $X$ contains the homogeneous elements
\[
 S_1 := -\frac{f_0}{T_2},
 \quad
 S_2 := -\frac{f_1+S_1T_1}{T_2},
 \quad
 \dots
 \quad
 S_d := -\frac{f_{d-1}+S_{d-1}T_1}{T_2}
\]
because, as shown in the proof of Theorem~\ref{main-teo}, each of them belongs to $R_{T_1}\cap R_{T_2}$.
\end{remark}

The following example shows that the presentation can change when
the codimension condition on the coefficients fails, even when the
hypersurface remains smooth. It does not isolate the generality
assumption, since that codimension condition also fails in the
special member considered below.

\begin{example}
Let ${\rm Bl}_p\mathbb P^4$ be the blowing-up
of $\mathbb P^4$ at an invariant point $p$
with exceptional divisor $E$ and pullback
of a general hyperplane $H$.
Let $X\subseteq {\rm Bl}_p\mathbb P^4$ be
a smooth hypersurface in $|5H-3E|$.
A defining equation for $X$ is
\[
 f_3T_5^2 - f_4T_5T_6 + f_5T_6^2 = 0,
\]
where $f_i\in\mathbb K[T_1,T_2,T_3,T_4]$
is homogeneous of degree $i$, for any $i=3,4,5$.
In the general case, by Theorem~\ref{main-teo} the
Cox ring and grading matrix of $X$ are

{\small
\[
 \frac{\mathbb K[T_1,\dots,T_6,S_1,S_2]}
 {\langle f_3+T_6S_2, f_4+T_5S_2+T_6S_1, f_5 + T_5S_1\rangle}
 \qquad
\begin{bmatrix}
  \hfill 1 & \hfill 1 & \hfill 1 & \hfill 1 & \hfill 1 & \hfill 0 & \hfill 4 & \hfill 3\\
  \hfill -1 & \hfill -1 & \hfill -1 & \hfill -1 & \hfill 0 & \hfill 1 & \hfill -5 & \hfill -4
\end{bmatrix}
\]
}

If we consider also particular cases, the number of
generators can decrease. For instance,
if $f_4$ vanishes identically and $f_3,f_5$ are general, the ideal
$(T_5,T_6,f_3,f_4,f_5)$ has height four rather than the required five.
In this case the Cox ring and grading matrix of $X$ are

{\small
\[
 \frac{\mathbb K[T_1,\dots,T_6,S_1]}
 {\langle f_3+T_6^2S_1, f_5 - T_5^2S_1\rangle}
 \qquad
\begin{bmatrix}
  \hfill 1 & \hfill 1 & \hfill 1 & \hfill 1 & \hfill 1 & \hfill 0 & \hfill 3 \\
  \hfill -1 & \hfill -1 & \hfill -1 & \hfill -1 & \hfill 0 & \hfill 1 & \hfill -5
\end{bmatrix}.
\]
}

To justify this presentation, set
$S=-f_3/T_6^2=f_5/T_5^2$ on the two corresponding affine charts.
The displayed quotient has these same localizations. It is a complete
intersection, and its locus $T_5=T_6=0$, defined also by $f_3=f_5=0$,
has codimension two. The chart normality follows as above from
smoothness of $X$ and the codimension of the remaining irrelevant
locus. Hence the quotient is normal and equals the intersection of
the two localizations. Smoothness and ampleness of $X$ in the smooth
fourfold give the required Picard-group isomorphism by Lefschetz.
\end{example}

\section{Proof of Corollary~\ref{main-cor}}
\label{sec:proof2}
Let us briefly recall
the classification of smooth projective
toric varieties with Picard rank two.
The following classification result, in the
complete case,
has been proved in~\cite{kl}*{Theorem 1
and Theorem 2} using the language
of fans. Here we provide
a shorter argument via the Cox ring
description.

\begin{proposition}
\label{prop:fano}
Let $\mathbb P$ be a smooth projective
toric variety of Picard rank two.
Then the grading matrix of the Cox ring
${\mathbb K}[T_1,\dots,T_{n+2}]$ of $\mathbb P$
and the irrelevant ideal are
\[
 \begin{bmatrix}
  1 & \cdots & 1 & 0 & -a_1 & \cdots & -a_k\\
  \undermat{n-k+1}{0 & \cdots & 0} & 1 & 1 & \cdots & 1
 \end{bmatrix},
 \qquad
 \langle T_1,\dots,T_{n-k+1}\rangle
 \cap
 \langle T_{n-k+2},\dots,T_{n+2}\rangle\\[10pt]
\]
where $1\leq k\leq n-1$,
and $0\leq a_1\leq\dots\leq a_k$.
Moreover $\mathbb P$ is Fano if and only if
$\sum_ia_i < n-k+1$.
\end{proposition}
\begin{proof}
Let $w_i = \deg(T_i)$ for any $i$
and let $v_i$
be the corresponding primitive
generator of the one dimensional
cone of the fan $\Sigma =\Sigma(\mathbb P)$.
Let $p_\sigma\in \mathbb P$
be the invariant point defined by
a maximal cone $\sigma\in\Sigma$.
Since $p_\sigma$ is a smooth point of
$\mathbb P$ it is a factorial point, that is the
local class group ${\rm Cl}(\mathbb P,p_\sigma)$ is
trivial (in a toric variety being smooth
and being factorial is the same).
By~\cite{adhl}*{Proposition 3.3.1.5}
we have
\[
 {\rm Cl}(\mathbb P,p_\sigma)\simeq
 {\rm Cl}(\mathbb P)/\langle w_i\, :\, v_i\notin\sigma(1)\rangle.
\]
The free abelian group in the above denominator
is generated by the two classes $w_i,w_j$
such that ${\mathbb Q}_{\geq 0}\cdot v_i$ and
${\mathbb Q}_{\geq 0}\cdot v_j$ are not one
dimensional rays of $\sigma$.
Thus the factoriality of $\mathbb P$ at $p_\sigma$
is equivalent to ask for the matrix $(w_i,w_j)$
to have determinant $\pm 1$.

Since $\mathbb P$ is smooth projective
its nef cone is full dimensional.
Let $w$ be an ample class. Then
by~\cite{adhl}*{Corollary 1.6.3.6}
the irrelevant ideal is generated by
the monomials $T_iT_j$ such that
$w\in{\rm cone}(w_i,w_j)$. Equivalently
the complement $\{1,\dots,n+2\}\setminus\{i,j\}$
of any such pair of indices $i,j$ defines
a maximal cone of $\Sigma$. Thus
by the previous argument we have
that each such matrix $(w_i,w_j)$
has determinant $\pm 1$.
In particular we can assume without loss of
generality that the nef cone of $\mathbb P$
is the first quadrant.
If there are other rays different from $(1,0)$ and
$(0,1)$, by smoothness hypothesis they must
be of the form $(-a,1)$ or $(1,-b)$, for $a,b > 0$.
Since the variety is complete the effective cone
is pointed so that, if there exists at least
one ray $(-a,1)$, there can be no rays of the form
$(1,-b)$. So up to exchanging the axes
we can assume that there are no
rays of the form $(1,-b)$ and in particular
the ray $(1,0)$ is an extremal ray of the
effective, moving and nef cone at the
same time. The situation is summarized
in the following picture.
 
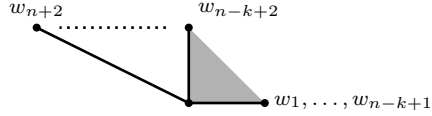
\begin{center}
 \begin{tikzpicture}[scale=1]
 \draw[line width=0mm,fill=gray!60!white] (0,1) -- (0,0) -- (1,0);
 \foreach \x/\y in {1/0,0/1}
 {\draw [line width=1pt] (0,0)-- (\x,\y);
 \draw [color=black,line width=1pt,fill=black] (\x,\y) circle (1pt);}
 \draw [color=black,line width=1pt,fill=black] (0,0) circle (1pt);
 \draw (0,1) node[above right] {{\footnotesize $w_{n-k+2}$}};
 \draw (1,0) node[right] {{\footnotesize $w_1,\dots, w_{n-k+1}$}};
 \draw [line width=1pt] (0,0)-- (-2,1);
  \draw [color=black,line width=1pt,fill=black] (-2,1) circle (1pt);
 \draw [line width=1pt,dotted] (-0.3,1)-- (-1.7,1);
 \draw (-2,1) node[above] {{\footnotesize $w_{n+2}$}};
 \end{tikzpicture}
 \end{center}
Since the $(1,0)$ is an extremal ray of the
moving cone, by~\cite{adhl}*{Proposition 3.3.2.3}
we deduce that $n-k+1\geq 2$. For the same reason
we have $k\geq 1$.
Finally $\mathbb P$ is Fano if and only if
the sum $\sum_iw_i$ lies in the
interior of the nef cone, equivalently
if $\sum_ia_i < n-k+1$.
\end{proof}

Observe that a smooth projective toric variety
$\mathbb P$ has a component of the irrelevant
locus of codimension $2$ when either $k=n-1$ or
$k=1$, with the same notation of
Proposition~\ref{prop:fano}.
The grading matrices for these two cases are
those which appear in the statement of
Corollary~\ref{main-cor}.

\begin{proof}[Proof of Corollary~\ref{main-cor}]
Let $f$ be a defining polynomial for $X$.

We prove case $(1)$.
Since the degree of $f$ is $[a,b]$, in Cox coordinates
we can write
\[
 f = \sum_{i=0}^a (-1)^if_iT_1^{a-i}T_2^i,
\]
where $f_0,\dots,f_a$ are homogeneous of degree
$[0,b]$. If we set $m:=\#\{i\in\{1,\dots,n-1\}:a_i=0\}$,
when we evaluate $f_0,\dots,f_a$ in $T_1=T_2=0$,
we get $a+1$ general homogeneous polynomials
of degree $b$ in the $m+1$ variables
$T_3,\dots,T_{m+3}$. Therefore the
condition $a\leq m$
guarantees that the codimension of $V(T_1,T_2,f_0,
\dots,f_a)$ is at least $a+3$, so that $f$
satisfies the hypotheses of
Theorem~\ref{main-teo}.

We prove $(2)$.
In this case we can write
\[
 f = \sum_{i=0}^b(-1)^if_iT_{n+1}^{b-i}T_{n+2}^i,
\]
where $f_i$ is homogeneous of degree
$a+ia_1$ in the variables $T_1,\dots,T_n$, for
any $0\leq i\leq b$. Under the hypothesis
$b\leq n-1$, these general $b+1$ polynomials
impose independent conditions, so that
the codimension of $V(T_{n+1},T_{n+2},f_0,\dots,f_b)$
is $b+3$. Therefore also in this case $f$
satisfies the hypotheses of
Theorem~\ref{main-teo}.
\end{proof}

\begin{remark}
 \label{rem:case2}
 Let us consider case (2) of Corollary~\ref{main-cor}, with
 $n=4$ and $a_1 = 1$, so that
 the toric variety $\mathbb P$ is
 the blowing-up of $\mathbb P^4$ at one point $p$, and
 a hypersurface $X\subseteq Z$ of degree
 $[a,b]$ corresponds
 to the strict transform of a hypersurface of degree
 $a+b$ in $\mathbb P^4$, having multiplicity $a$
 at the point $p$. If we fix $b=4$ and $a=1$, we
 have that $X$ can be described as the blow-up
 in one point of a hypersurface of degree $5$ in
 $\mathbb P^4$. In particular this example can be treated using the techniques of~\cite{hlk}, looking for generators
 corresponding to surfaces in $X$, having
 bounded multiplicity at the point $p$.
 With the help of the algebra computer
 system Magma~\cite{magma} we found $223$
 generators with multiplicity at most $21$.
 This may suggest that in this case the
 Cox ring is not finitely generated (even if
 we are not able to prove that).
 \end{remark}

\begin{remark}
\label{rem:P1n}
When all the twists in Corollary~\ref{main-cor} vanish, the ambient
variety is $\mathbb P^1\times\mathbb P^{n-1}$, after exchanging the
factors when necessary. Thus the Cox ring of a general hypersurface
of bidegree $[a,b]$, with $1\leq a\leq n-1$ and $b>0$, has the
presentation of Theorem~\ref{main-teo}.
This agrees with \cite{ott}*{Thm.~1.1}.
For a very general hypersurface with $a\geq n$ and $b\geq2$, the
effective cone is not closed and the Cox ring is not finitely
generated; see \cite{ott}*{Thm.~5.6}.
The case $b=1$ is an exception: the hypersurface is a projective
bundle over $\mathbb P^1$, hence toric. The next proposition gives
its precise isomorphism type.
\end{remark}

\begin{proposition}\label{prop:scroll}
Let $X = V(f) \subseteq {\mathbb P}^1\times
{\mathbb P}^{n-1}$, with $n\geq 4$, where $f$ is a general
homogeneous polynomial of degree $[d,1]$, with $d\geq0$. Let
$a\in\{0,\dots,n-2\}$ be the residue of $d$ modulo $n-1$.
Then $X$ is the scroll
\[
    S(\undermat{n-a-1}{0,\dots,0},\undermat{a}{1,\dots,1}).\
\]

In particular, if $(n-1)\mid d$ then $X \simeq {\mathbb P}^1\times
    {\mathbb P}^{n-2}$.
\end{proposition}
\begin{proof}
Put $r=n-1$ and write
\[
 f=\sum_{i=0}^d(-1)^if_i s^{d-i}t^i,
\]
where the $f_i$ are general linear forms in $y_0,\dots,y_r$.
Set $A=\mathbb K[s,t,y_0,\dots,y_r]/(f)$.
By Theorem~\ref{mteo}, we compute
$\mathcal R(X)=A_s\cap A_t$ using the fraction adjunctions
of Algorithm~\ref{alg}. Normalization, whenever required,
does not change these localizations.

The successive fractions
\[
 z_1=-\frac{f_0}{t},\qquad
 z_{i+1}=-\frac{f_i+sz_i}{t}\quad(1\leq i<d)
\]
belong to $A_s\cap A_t$: the equation $f=0$ also expresses
them with powers of $s$ as denominators. Thus they are
admissible new generators in the algorithm, all of degree
$(-1,1)$.

If $1\leq d\leq r$, adjoining $z_1,\dots,z_d$ gives
\[
 f_0+tz_1=0,\qquad
 f_i+sz_i+tz_{i+1}=0\quad(1\leq i<d),\qquad
 f_d+sz_d=0.
\]
Since $f_0,\dots,f_d$ are linearly independent, these
relations eliminate $d+1$ of the $y$-variables.
The resulting ring is polynomial, with $r-d$ generators
of degree $(0,1)$ and $d$ generators of degree $(-1,1)$,
besides $s,t$. Its localizations at $s$ and $t$ agree
with those of $A$, and $V(s,t)$ has codimension two,
so the algorithm terminates with this ring.
For $d=0$, simply eliminate $f_0$.

Suppose now that $d>r$ and put $m=d-r$.
After a linear change of the $y$-variables, assume
$f_i=y_i$ for $0\leq i\leq r$.
Adjoin $z_1,\dots,z_n$ and use
\[
 y_0=-tz_1,\qquad
 y_i=-sz_i-tz_{i+1}\quad(1\leq i\leq r)
\]
to eliminate all the $y$-variables.
Substitution in $f$ gives $f=(-1)^nt^ng$, where
\[
 g=s^mz_n+
 \sum_{j=0}^{m-1}(-1)^j
 f_{n+j}(-tz_1,-sz_1-tz_2,\dots,-sz_r-tz_n)
 s^{m-1-j}t^j.
\]
Hence $B:=\mathbb K[s,t,z_1,\dots,z_n]/(g)$ satisfies
$B_s=A_s$ and $B_t=A_t$.
Indeed, the substitutions have inverses on both charts:
on $D(t)$ use the recursive fractions; on $D(s)$,
the relation
\[
 s^mz_n=
 -\sum_{j=0}^{m-1}(-1)^jf_{n+j}(y)s^{m-1-j}t^j
\]
determines $z_n$, and the preceding relations then
determine $z_r,\dots,z_1$ backwards.
They also identify the corresponding projective
hypersurfaces over $\mathbb P^1$.
In the basis where $\deg z_i=(0,1)$, the new equation
has bidegree $(m,1)$.

We verify that $g$ is general up to a linear change
of its $z$-variables, so that this reduction can be iterated.
Let $N$ be the $n\times n$ matrix with ones immediately
below the diagonal and zeros elsewhere, and let
$e_n^{\mathsf T}=(0,\dots,0,1)$.
The row of coefficients of $g$ has the form
\[
 s^me_n^{\mathsf T}-h(s,t)(tI+sN),
\]
where $h$ is an arbitrary row of $n$ binary forms
of degree $m-1$.
Every row $G$ of binary forms of degree $m$ has a unique
division
\[
 G(s,t)=H(s,t)(tI+sN)+s^mR,
\]
with $R$ a constant row.
For general $G$, the last entry of $R$ is nonzero.
There is then an invertible matrix $C$, polynomial
in $N$, such that $RC=e_n^{\mathsf T}$.
Since $C$ commutes with $N$,
\[
 GC=s^me_n^{\mathsf T}+(HC)(tI+sN).
\]
Thus every general coefficient row is linearly equivalent
to one in the displayed family.
In particular, this family meets every nonempty open
condition invariant under linear changes of the fiber
coordinates. This proves the generality needed for induction.

Each reduction lowers $d$ by $r$ and replaces the fiber
generators by generators whose first degree is one less.
Writing $d=qr+a$, with $0\leq a<r$, iteration and the
base case therefore give
\[
 \mathcal R(X)=\mathbb K[s,t,u_1,\dots,u_r],\qquad
 \deg u_i=
 \begin{cases}
 (-q,1),&1\leq i\leq r-a,\\
 (-q-1,1),&r-a<i\leq r.
 \end{cases}
\]
After the change of basis
$(\alpha,\beta)\mapsto(\alpha+q\beta,\beta)$,
these degrees are $(0,1)$ and $(-1,1)$, respectively.
Together with the two base charts above, this is
the Cox presentation of
\[
 S(\undermat{r-a}{0,\dots,0},\undermat{a}{1,\dots,1}).
\]
If $a=0$, it is $\mathbb P^1\times\mathbb P^{n-2}$.
\end{proof}

\section{Applications}
\label{sec:applications}

Our first application of Corollary~\ref{main-cor}
is to give a presentation for the
Cox ring of general Calabi-Yau anticanonical
hypersurfaces in smooth
toric Fano varieties of Picard rank two.

\begin{corollary}
\label{mt}
Let $\mathbb P$ be a smooth toric Fano variety
of dimension $n\geq 4$ and Picard rank two.
Let $X\subseteq\mathbb P$ be a general anticanonical
hypersurface whose homogeneous equation
is $f = 0$.
Then the Cox ring of $X$
and its grading matrix are given in the
following table.

{\footnotesize
\begin{longtable}{ll}
Cox ring & Grading matrix \\
\hline
\\
${\mathbb K}[T_1,\dots,T_{n+2}]/\left\langle f\right\rangle$
&
$
 \left[
 \begin{array}{ccccccc}
  1 & \cdots & 1 & 0 & -a_1 & \cdots & -a_k\\
  \undermat{n-k+1}{0 & \cdots & 0} & 1 & 1 & \cdots & 1
 \end{array}
 \right]
$\\
\\
\\
\multicolumn{2}{l}{
where $2\leq k\leq n-2$, $0\leq a_1\leq\cdots\leq a_k$,
and $\sum_ia_i<n-k+1$}
\\
\\
\hline
\\
$
\begin{array}{l}
{\mathbb K}[T_1,\dots,T_{n+2},S]/I
\\[1ex]
\text{with $I$ generated by}
\\[1ex]
f_1 + ST_1,\\
f_0 + ST_2
\end{array}
$
&
$
 \left[
 \begin{array}{cccccr|r}
   1  & 1 & 0 & \cdots &  0 & -1 & -1\\
  0  & 0 & \undermat{n-1}{1 & \cdots &  1} &  1 & n
 \end{array}
 \right]
$\\
\\
\multicolumn{2}{l}{
where $f = f_0T_1 - f_1T_2$}
\\
\\
\hline
\\
$
\begin{array}{l}
{\mathbb K}[T_1,\dots,T_{n+2},S_1,S_2]/I
\\[1ex]
\text{with $I$ generated by}
\\[1ex]
 f_0 + T_{n+2}S_2,\\
 f_1 + T_{n+1}S_2+T_{n+2}S_1,\\
 f_2 + T_{n+1}S_1
 \end{array}
$
&
$
 \left[
 \begin{array}{ccccc|cr}
  1 & \cdots & 1 & 0 & -a_1 & n+a_1 & n\\
  \undermat{n}{0 & \cdots & 0} & 1 & 1 & -1 & -1
 \end{array}
 \right]
$\\
\\
\multicolumn{2}{l}{
where $0\leq a_1\leq n-1$ and
$f = f_0T_{n+1}^2 - f_1T_{n+1}T_{n+2} + f_2T_{n+2}^2$}
\end{longtable}
}
\end{corollary}

\begin{proof}
The grading matrix and irrelevant ideal of
$\mathbb P$ are given in Proposition~\ref{prop:fano},
and in what follows we use the same notation.

When $2\leq k\leq n-2$ the irrelevant locus of
$\mathbb P$ has no components of codimension $2$
so that the Cox ring is a hypersurface (first line
of the above table).

When $k = n-1$, the ambient toric variety is Fano
exactly when $a_1=\cdots=a_{n-2}=0$ and
$0\leq a_{n-1}\leq 1$. If $a_{n-1} = 1$,
then the anticanonical hypersurface $X$
has degree $[1,n]$ and, by Corollary~\ref{main-cor},
the Cox ring is the one given in Theorem~\ref{main-teo}
with $d = 1$.
The case $a_{n-1} = 0$ corresponds to
the ambient toric variety $\mathbb P^1
\times\mathbb P^{n-1}$, which is analyzed
in the next paragraph.

When $k=1$, the ambient toric variety is Fano
exactly when $a_1 \leq n-1$.
The anticanonical hypersurface $X$ has degree
$[n-a_1,2]$ and, by Corollary~\ref{main-cor},
the Cox ring is the one given in Theorem~\ref{main-teo}
with $d = 2$.

\end{proof}

We now consider the case of a surface
$X$ embedded into a smooth projective toric
threefold $\mathbb P$ of Picard rank $2$.
Since the Cox ring of $\mathbb P$ has
$5$ generators it follows that the irrelevant
locus is always the union of a codimension
$2$ with a codimension $3$ linear subspaces.
In what follows we compute the Cox ring
of $X$ in some cases, using the same notation
of Proposition~\ref{prop:fano}.
\begin{corollary}
\label{rk2-sur}
Let $\mathbb P$ be a smooth projective
toric threefold of Picard rank two.
Let $X\subseteq\mathbb P$ be a very
general ample surface whose
homogeneous equation is $f = 0$.
In the following table we compute
a presentation for the Cox ring of
certain such $X$.

{\footnotesize
\begin{longtable}{lll}
 & Cox ring & Grading matrix \\
\hline
\\
$
\begin{array}{l}
k = 2,\\
a_1=0,\quad a_2 > 0,\\
\deg(f) = [1,b]\\
b\geq 3
\end{array}
$
&
$
\begin{array}{l}
{\mathbb K}[T_1,\dots,T_5,S]/I
\\[1ex]
\text{with $I$ generated by}
\\[1ex]
f_1 + ST_1,\\
f_0 + ST_2
\end{array}
$
&
$
 \left[
 \begin{array}{ccccc|r}
  1  & 1 & 0 & 0 & -a_2 & -1\\
  0  & 0 & 1 & 1 &  1 & b
 \end{array}
 \right]
$\\
\\
&\multicolumn{2}{l}{
where $f = f_0T_1 - f_1T_2$}\\
\\
\hline
\\
$
\begin{array}{l}
k = 1,\\
a_1 \geq 0,\\
\deg(f) = [a,2]\\
a\geq \max\{1,3-a_1\}
\end{array}
$
&
$
\begin{array}{l}
{\mathbb K}[T_1,\dots,T_5,S_1,S_2]/I
\\[1ex]
\text{with $I$ generated by}
\\[1ex]
 f_0 + T_5S_2,\\
 f_1 + T_4S_2+T_5S_1,\\
 f_2 + T_4S_1
 \end{array}
$
&
$
 \left[
 \begin{array}{ccccc|cc}
  1 & 1 & 1 & 0 & -a_1 & a+2a_1 & a+a_1\\
  0 & 0 & 0 & 1 & 1 & -1 & -1
 \end{array}
 \right]
$\\
\\
&\multicolumn{2}{l}{
where $f = f_0T_4^2 - f_1T_4T_5 + f_2T_5^2$}
\\
\end{longtable}
}
\end{corollary}
\begin{proof}
By Proposition~\ref{prop:fano},
the grading matrix of the Cox ring
of $\mathbb{P}$ is one of the following
(corresponding to the cases $k=2$ and $k=1$
respectively):
\[
 \left[
 \begin{array}{ccccc}
  1  & 1 & 0 & -a_1 & -a_2\\
  0  & 0 & 1 & 1 &  1
 \end{array}
 \right]
 \qquad
 \left[
 \begin{array}{ccccc}
  1  & 1 & 1 & 0 & -a_1\\
  0  & 0 & 0 & 1 & 1
 \end{array}
 \right],
\]
where $0\leq a_1\leq a_2$.
The nef cone of $\mathbb{P}$ is the positive quadrant in both cases thus $f$
has one of the following form:
\[
 f = \sum_{i=0}^d(-1)^if_iT_1^{d-i}T_2^i,
 \qquad
 f = \sum_{i=0}^d(-1)^if_iT_4^{d-i}T_5^i,
\]
respectively. Observe that since the spectrum
of the Cox ring of $\mathbb P$ has dimension $5$,
the codimension condition given in
Theorem~\ref{main-teo} implies $d\leq 2$.

We proceed as in Theorem~\ref{main-teo}, whose proof applies verbatim here once the isomorphism at the level of divisor class groups is established.
We use here the Lefschetz-type theorem~\cite[Thm. 1]{rs2}
whose hypotheses are: the class $H$ of $X$ must be ample and base point free, $H + K_{\mathbb{P}}$ must be base point free, and $X \in |H|$ must be very general. Since $\mathbb{P}$ is smooth and toric, the first two conditions are equivalent to requiring that $H$ is just ample
(i.e. the class of $H$ is $[a,b]$, with $a,b > 0$)
and $H + K_{\mathbb{P}}$ is nef.

Putting all this together, in case
$k=2$, since $d = a$ we obtain that
\[
 0<a\leq 2,
 \qquad
 b>0,
 \qquad
 a+a_1+a_2-2\geq 0,
 \qquad
 b-3\geq 0.
\]
When we evaluate $f_0,\dots,f_a$ in $T_1=T_2=0$,
we get $a+1$ general homogeneous polynomials
of degree $b$ in the variables $T_3$ (if $a_1>0$),
$T_3,T_4$ (if $a_1=0$ and $a_2>0$) or
$T_3,T_4,T_5$ (if $a_1=a_2=0$).
In the first case the codimension condition
cannot be satisfied. In the second case it
can be satisfied only if $a = d = 1$.
The third case has been included in
the analysis of $k=1$, after a change
of variables.

Analogously, when $k=1$ we have that $d = b$, and
we get the conditions:
\[
 a>0,
 \qquad
 0 < b \leq 2,
 \qquad
 a+a_1-3\geq 0,
 \qquad
 b-2\geq 0,
\]
which are the ones displayed in the statement.
We conclude observing that under these conditions
the locus $V(T_4,T_5,f_0,f_1,f_2)$ has
the required codimension $5$.
\end{proof}

\section{Examples of Picard rank 3}
\label{sec:rank3}

\begin{example}
\label{ex-1}
 Let $Z := {\rm Bl}_2\mathbb P^4$ be the
 blowing up of $\mathbb P^4$ in two points,
 and let us denote by $H$ the class of the
 pull back of a hyperplane and by
 $E_1,\, E_2$ the exceptional divisors.
 The grading matrix
 of the Cox ring $\mathcal R(Z) =
 \mathbb K[T_1,\dots,T_7]$ is the following
\[
\left[
 \begin{array}{rrrrrrr}
  1 & 1 & 1 & 1 & 1 & 0 & 0\\
  -1 & -1 & -1 & -1 & 0 & 1 & 0\\
  -1 & -1 & -1 & 0 & -1 & 0 & 1\\
  \end{array}
  \right].
\]
The codimension-two part of the irrelevant locus is defined by
\[
\langle T_4,T_7\rangle\cap\langle T_5,T_6\rangle\cap\langle T_6,T_7\rangle.
\]
Equivalently, for the localization computation, we use the monomial ideal
\[
\langle T_6T_7, T_4T_6, T_5T_7 \rangle.
\]
In the following picture
we display the
effective cone of $\mathbb P$, the nef cone (in dark grey), and the movable one (the union of the
nef cone and the light grey triangle).

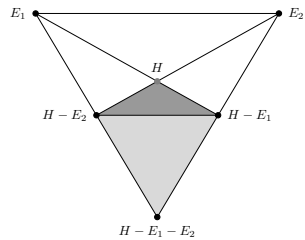
\begin{center}
\begin{tikzpicture}[xscale=.4,
yscale=.45]
\tkzDefPoint(0,0){O}
\tkzDefPoint(-4,2){A}
\tkzDefPoint(4,2){B}
\tkzDefPoint(-2,-1){C}
\tkzDefPoint(2,-1){D}
\tkzDefPoint(0,-4){E}

\fill[gray!80] (O) -- (C) -- (D) -- cycle;
\fill[gray!30] (E) -- (C) -- (D) -- cycle;

\tkzDrawSegments(A,E A,B B,E B,C A,D C,D)
\tkzDrawPoints[size=2](A,B,C,D,E)
\tkzDrawPoints[size=2, color=gray](O)
\tkzLabelPoint[left](A){\scalebox{.5}{$E_1$}}
\tkzLabelPoint[right](B){\scalebox{.5}{$E_2$}}
\tkzLabelPoint[above](O){\scalebox{.5}{$H$}}
\tkzLabelPoint[left](C){\scalebox{.5}{$H-E_2$}}
\tkzLabelPoint[right](D){\scalebox{.5}{$H-E_1$}}
\tkzLabelPoint[below](E){\scalebox{.5}{$H-E_1-E_2$}}
\end{tikzpicture}
\end{center}

Let us consider a general hypersurface
$X\subseteq Z$ in an ample degree, for
instance $3H-E_1-E_2$. An equation $f$
of $X$ can be written as
{\small
\[
 T_4^2T_5T_6 + a_1T_4^2T_6^2 + T_4T_5^2T_7 + b_1T_4T_5T_6T_7 + b_2T_4T_6^2T_7 +
    c_1T_5^2T_7^2 + a_2T_5T_6T_7^2 + a_3T_6^2T_7^2
\]}
where $a_i,b_i,c_i\in\mathbb K[T_1,T_2,T_3]$
are general homogeneous of degree $i$.
Searching for fractions with denominators that are powers of
$m_1:=T_6T_7$ yields the following elements of the Cox ring
\footnote{The Magma program to find the four elements is available at
\url{https://github.com/alaface/Cox-subvarieties/blob/main/Example\%206.1}.}
\[
\begin{array}{ll}
 S_1 :=  \dfrac{a_1T_6 + T_5}{T_7},
 &
 S_2 :=  \dfrac{c_1T_7 + T_4}{T_6},\\[2mm]
 S_3 := \dfrac{a_2 - b_1c_1 + c_1^2 + S_1S_2}
 {T_6},
 & S_4 := \dfrac{a_1^2 - a_1b_1 + b_2 + S_1S_2}{T_7}.
 \end{array}
\]
Their ideal of relations in
$\mathbb K[T_1,\dots,T_7,S_1,\dots,S_4]$
is the following complete intersection
{\small
\[
    \tilde J = \left\langle
    \begin{array}{cc}
    a_1 T_6 + T_5 - T_7 S_1, & c_1 T_7 + T_4 - T_6 S_2 \\
    -T_6 S_3 + a_2 - b_1 c_1 + c_1^2 + S_1 S_2,
    & -T_7 S_4 + a_1^2 - a_1 b_1 + b_2 + S_1 S_2 \\
    \multicolumn{2}{l}{-a_1 T_6 S_3 + T_6 S_2 S_4 + T_7 S_1 S_3 + a_1 c_1^2 - a_1 S_1 S_2 + a_3 + b_1 S_1 S_2 - b_2 c_1 - 2 c_1 S_1 S_2}
\end{array}
    \right\rangle.
\]}
We verify that this quotient is normal and that the displayed
fractions generate the whole Cox ring. Eliminate $T_4,T_5$ using
the first two relations, and put
\[
\begin{gathered}
y=T_6,\quad z=T_7,\quad
A=a_2-b_1c_1+c_1^2,\quad B=a_1^2-a_1b_1+b_2,\\
L=b_1-2a_1-2c_1,\qquad
D=a_3-a_1a_2+a_1b_1c_1-b_2c_1.
\end{gathered}
\]
The remaining relations are equivalent to
\[
ys_3=A+s_1s_2,\qquad zs_4=B+s_1s_2,\qquad
ys_2s_4+zs_1s_3+Ls_1s_2+D=0,
\]
where $s_i=S_i$. For general coefficients, each of
$(A,D)$, $(B,D)$ and $(B-A,D-LA)$ has height two in
$\mathbb K[T_1,T_2,T_3]$.
On $y=z=0$ the equations reduce to
$B-A=D-LA=0$ and $s_1s_2=-A$, with $s_3,s_4$ free;
hence this locus has dimension four.
On $T_4=z=0$ we have $ys_2=0$. The branch $y=0$ is already
covered, while the branch $s_2=0$ is defined by
$B=D=0$ and $ys_3=A$, with $s_1,s_4$ free, and also has
dimension four. The locus $T_5=y=0$ is treated similarly,
using $A=D=0$ on its additional branch.

Outside the union of these three loci, elimination identifies the
quotient with the corresponding localizations of
$\mathbb K[T_1,\dots,T_7]/(f)$; each has dimension six.
Every minimal component of the quotient by the three remaining
equations has dimension at least six, whereas the omitted loci have
dimension four. Thus no component is supported there, and the
quotient is a complete intersection of dimension six.
The common open subset is irreducible and reduced, so the
Cohen--Macaulay quotient is a domain.
Its charts are normal: as in the proof of
Theorem~\ref{main-teo}, Bertini gives smoothness on the characteristic
space, and the remaining irrelevant subsets have codimension at
least two. The omitted loci have codimension two in the new
quotient, so Serre's criterion proves that it is normal.
Lemma~\ref{pro:crit} and Theorem~\ref{mteo} now show that the four
fractions generate $\mathcal R(X)$.
Eliminating
the variables $T_4$ and $T_5$, by means of the
first two equations, and putting
$J := \tilde J \cap \mathbb K[T_1,T_2,T_3,T_6,T_7,S_1,\dots,S_4]$
the Cox ring $\mathcal R(X)$
and the grading matrix are
\[
\mathbb K[T_1,T_2,T_3,T_6,T_7,S_1,\dots,S_4]
/J,
\quad
\left[
\begin{array}{rrrrrrrrr}
1 & 1 & 1 & 0 & 0 & 1 & 1 & 2 & 2\\
-1 & -1 & -1 & 1 & 0 & 0 & -2 & -3 & -2\\
-1 & -1 & -1 & 0 & 1 & -2 & 0 & -2 & -3
\end{array}
\right].
\]
In particular, $\mathcal R(X)$ is a complete intersection.
The following picture displays its effective, movable and nef cones.

\begin{center}
\begin{tikzpicture}[xscale=.4,
yscale=.45]

\begin{scope}[scale = .7]

\tkzDefPoint(0,0){O}
\tkzDefPoint(-4,2){A}
\tkzDefPoint(4,2){B}
\tkzDefPoint(-2,-1){C}
\tkzDefPoint(2,-1){D}
\tkzDefPoint(0,-4){E}
\tkzDefPoint(8,-4){F}
\tkzDefPoint(-8,-4){G}
\tkzDefPoint(-4,-10){H}
\tkzDefPoint(4,-10){I}
\tkzDefPoint(0,-8){L}
\tkzDefPoint(-4,-2){M}
\tkzDefPoint(4,-2){N}
\tkzDefPoint(-4,-6){P}
\tkzDefPoint(4,-6){Q}

\tkzDefPoint(-4,-4){R}
\tkzDefPoint(4,-4){S}
\tkzDefPoint(-2,-7){T}
\tkzDefPoint(2,-7){U}

\fill[gray!30] (O) -- (M) -- (P) -- (L) -- (Q) -- (N) -- cycle;
\fill[gray!80] (E) -- (D) -- (O) -- (C) -- cycle;

\tkzDrawSegments(A,B A,G G,H H,I I,F F,B)
\tkzDrawSegments(O,M M,P P,L L,Q Q,N N,O C,E D,E        A,F A,H B,G B,I F,H
G,I A,I B,F B,H F,G)
\tkzDrawPoints[size=2, color=gray](O,L,M,N,P,Q,C,D,R,S,T,U)
\tkzDrawPoints[size=2, color=black](A,B,E,F,G,H,I)

\tkzLabelPoint[left](A){\scalebox{.5}{$E_1$}}
\tkzLabelPoint[right](B){\scalebox{.5}{$E_2$}}
\tkzLabelPoint[above](O){\scalebox{.5}{$H$}}
\tkzLabelPoint[below](E){\scalebox{.5}{$H-E_1-E_2$}}
\tkzLabelPoint[right](F){\scalebox{.5}{$H-2E_1$}}
\tkzLabelPoint[left](G){\scalebox{.5}{$H-2E_2$}}
\tkzLabelPoint[below](H){\scalebox{.5}{$2H-2E_1-3E_2$}}
\tkzLabelPoint[below](I){\scalebox{.5}{$2H-3E_1-2E_2$}}
\tkzLabelPoint[below](L){\scalebox{.5}{$5H-6E_1-6E_2$}}
\tkzLabelPoint[above left](M){\scalebox{.5}{$2H-3E_2$}}
\tkzLabelPoint[below left](P){\scalebox{.5}{$4H-3E_1-6E_2$}}
\tkzLabelPoint[above right](N){\scalebox{.5}{$2H-3E_1$}}
\tkzLabelPoint[below right](Q){\scalebox{.5}{$4H-6E_1-3E_2$}}
\end{scope}

\end{tikzpicture}
\end{center}

\end{example}

\begin{example}\label{ex-2}
If we now pass to the toric model obtained by the anti-flip of the strict
transform of the line through the two points
of $\mathbb P^4$, we obtain a new toric
$4$-fold $Z'$. The codimension-two part of the irrelevant locus is defined by
\[
\langle T_4,T_7\rangle\cap\langle T_5,T_6\rangle.
\]
Equivalently, for the localization computation, we use the monomial ideal
\[
\langle T_4T_5, T_4T_6, T_5T_7, T_6T_7 \rangle.
\]
In the following picture
we display the
effective cone of $\mathbb P$, the nef cone (in dark grey), and the movable one (the union of the
nef cone and the light grey triangle).

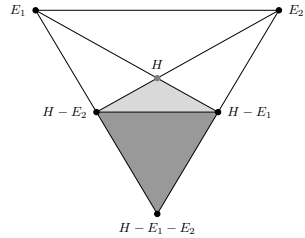
\begin{center}
\begin{tikzpicture}[xscale=.4,
yscale=.45]
\tkzDefPoint(0,0){O}
\tkzDefPoint(-4,2){A}
\tkzDefPoint(4,2){B}
\tkzDefPoint(-2,-1){C}
\tkzDefPoint(2,-1){D}
\tkzDefPoint(0,-4){E}

\fill[gray!30] (O) -- (C) -- (D) -- cycle;
\fill[gray!80] (E) -- (C) -- (D) -- cycle;

\tkzDrawSegments(A,E A,B B,E B,C A,D C,D)
\tkzDrawPoints[size=2](A,B,C,D,E)
\tkzDrawPoints[size=2, color=gray](O)
\tkzLabelPoint[left](A){\scalebox{.5}{$E_1$}}
\tkzLabelPoint[right](B){\scalebox{.5}{$E_2$}}
\tkzLabelPoint[above](O){\scalebox{.5}{$H$}}
\tkzLabelPoint[left](C){\scalebox{.5}{$H-E_2$}}
\tkzLabelPoint[right](D){\scalebox{.5}{$H-E_1$}}
\tkzLabelPoint[below](E){\scalebox{.5}{$H-E_1-E_2$}}
\end{tikzpicture}
\end{center}

Let us consider a hypersurface
$X'\subseteq Z'$ in the ample degree
$3H-2E_1-2E_2$. An equation $f$
of $X'$ can be written as
\[
 a_1T_4T_5 +a_2T_4T_6+b_2T_5T_7+a_3T_6T_7
\]
where $a_i,b_i\in\mathbb K[T_1,T_2,T_3]$
are general homogeneous of degree $i$.
The localization computation using $T_4T_5$ yields
the following elements of the Cox ring
\footnote{The Magma program to find the two elements is available at \url{https://github.com/alaface/Cox-subvarieties/blob/main/Example\%206.2}.}
\[
\begin{array}{ll}
 S_1 :=  \dfrac{b_2T_5 + a_3T_6}{T_4},
 &
 S_2 :=  \dfrac{a_2T_4 + a_3T_7}{T_5}.
 \end{array}
\]
We claim that these elements generate the Cox ring and that its
presentation and grading matrix are
\[
\mathbb K[T_1,\dots,T_7,S_1,S_2]/J,
\quad
\left[
\begin{array}{rrrrrrrrr}
1 & 1 & 1 & 1 & 1 & 0 & 0 & 2 & 2\\
-1 &-1 &-1 &-1 & 0 & 1 & 0 &-1 &-3\\
 -1 &-1 &-1 & 0  &-1 & 0 & 1 &-3 &-1
\end{array}
\right]
\]
where
\[
    J = \left\langle
    \begin{array}{ll}
    a_1T_5 + a_2T_6 + T_7S_1, &
    T_4S_1 - b_2T_5 - a_3T_6, \\
    a_2T_4 - T_5S_2 + a_3T_7, &
    a_1T_4 + T_6S_2 + b_2T_7,\\
    a_1a_3 - a_2b_2 + S_1S_2,
    \end{array}
    \right\rangle.
\]
To prove the claim, write
$p=T_4$, $q=T_5$, $r=T_6$, $s=T_7$, $x=S_1$, $y=S_2$,
and $A=a_1$, $B=a_2$, $C=b_2$, $D=a_3$.
The five generators of $J$ are, up to order and sign, the
maximal Pfaffians of
\[
M=\begin{pmatrix}
0&A&B&y&s\\
-A&0&x&C&-r\\
-B&-x&0&D&q\\
-y&-C&-D&0&p\\
-s&r&-q&-p&0
\end{pmatrix}.
\]
Outside $V(p,s)\cup V(q,r)$, elimination identifies the quotient
with the corresponding localizations of the original hypersurface
ring, which are normal domains of dimension six.
We check the dimension of the omitted loci. On $p=s=0$, the equations
are
\[
Aq+Br=Cq+Dr=qy=ry=0,\qquad AD-BC+xy=0.
\]
If $y\neq0$, then $q=r=0$ and $x$ is determined, giving dimension
four. If $y=0$, the base lies on $AD-BC=0$; where the coefficient
matrix has rank one, its kernel in $(q,r)$ has dimension one and
$x$ is free, again giving dimension four. Its rank-zero locus is
$A=B=C=D=0$, which is just the origin in the three base variables
for general coefficients, and contributes dimension three.
The locus $q=r=0$ is analogous. Thus the omitted union has dimension
four.

The height bound for maximal Pfaffians of an alternating
$5\times5$ matrix excludes minimal components of codimension greater
than three. Consequently no component is supported on the omitted
union, and $J$ has height three.
The Buchsbaum--Eisenbud Pfaffian criterion
\cite{be} makes the quotient Cohen--Macaulay.
It has a unique minimal prime and is generically reduced by the
localization identifications, hence it is a domain.
Normality follows from Serre's criterion, since the omitted locus
has codimension two. Lemma~\ref{pro:crit} and
Theorem~\ref{mteo} now identify the quotient with $\mathcal R(X')$.
At the homogeneous origin every entry of $M$ belongs to the maximal
ideal, so its Pfaffian resolution is minimal. In particular, $J$
requires five generators although it has height three. Therefore
$\mathcal R(X')$ is finitely generated but is not a complete
intersection.
In the following picture we display the
cones ${\rm Eff}(X'),\,
{\rm Mov}(X')$ and ${\rm Nef}(X')$.

\begin{center}
\begin{tikzpicture}[xscale=.4,
yscale=.45]

\begin{scope}[xshift=15cm, scale=1]
\tkzDefPoint(0,0){O}
\tkzDefPoint(-4,2){A}
\tkzDefPoint(4,2){B}
\tkzDefPoint(-2,-1){C}
\tkzDefPoint(2,-1){D}
\tkzDefPoint(0,-4){E}
\tkzDefPoint(-4,-4){F}
\tkzDefPoint(4,-4){G}

\fill[gray!30] (O) -- (C) -- (D) -- cycle;
\fill[gray!80] (E) -- (C) -- (D) -- cycle;

\tkzDrawSegments(A,B C,D A,F F,G G,B O,C O,D D,E C,E)
\tkzDrawSegments(A,O O,B A,C B,D C,F D,G)
\tkzDrawPoints[size=2](A,B,C,D,E,F,G)
\tkzDrawPoints[size=2, color=gray](O)
\tkzLabelPoint[left](A){\scalebox{.5}{$E_1$}}
\tkzLabelPoint[right](B){\scalebox{.5}{$E_2$}}
\tkzLabelPoint[above](O){\scalebox{.5}{$H$}}
\tkzLabelPoint[left](C){\scalebox{.5}{$H-E_2$}}
\tkzLabelPoint[right](D){\scalebox{.5}{$H-E_1$}}
\tkzLabelPoint[below](E){\scalebox{.5}{$H-E_1-E_2$}}
\tkzLabelPoint[below](F){\scalebox{.5}{$2H-E_1-3E_2$}}
\tkzLabelPoint[below](G){\scalebox{.5}{$2H-3E_1-E_2$}}
\end{scope}
\end{tikzpicture}
\end{center}

\end{example}


\begin{bibdiv}
\begin{biblist}

\bib{abban2}{article}{
  author={Ahmadinezhad, Hamid},
  author={Zucconi, Francesco},
  title={Mori dream spaces and birational rigidity of Fano 3-folds},
  journal={Adv. Math.}, volume={292}, date={2016}, pages={410--445},
}
\bib{al}{article}{
  author={Artebani, Michela}, author={Laface, Antonio},
  title={Hypersurfaces in Mori dream spaces},
  journal={J. Algebra}, volume={371}, date={2012}, pages={26--37},
}
\bib{adhl}{book}{
  author={Arzhantsev, Ivan}, author={Derenthal, Ulrich},
  author={Hausen, J\"urgen}, author={Laface, Antonio},
  title={Cox rings}, series={Cambridge Studies in Advanced Mathematics},
  volume={144}, publisher={Cambridge University Press},
  address={Cambridge}, date={2015},
}
\bib{magma}{article}{
  author={Bosma, Wieb}, author={Cannon, John}, author={Playoust, Catherine},
  title={The Magma algebra system. I. The user language},
  journal={J. Symbolic Comput.}, volume={24}, date={1997}, pages={235--265},
}
\bib{be}{article}{
  author={Buchsbaum, David A.}, author={Eisenbud, David},
  title={Algebra structures for finite free resolutions, and some structure theorems for ideals of codimension 3},
  journal={Amer. J. Math.}, volume={99}, number={3}, date={1977}, pages={447--485},
}
\bib{cls}{book}{
  author={Cox, David A.}, author={Little, John B.}, author={Schenck, Henry K.},
  title={Toric varieties}, series={Graduate Studies in Mathematics},
  volume={124}, publisher={American Mathematical Society},
  address={Providence, RI}, date={2011},
}
\bib{denisi}{article}{
  author={Denisi, Francesco Antonio},
  title={Birational geometry of hypersurfaces in products of weighted projective spaces},
  journal={Math. Z.}, volume={313}, number={2}, date={2026},
  doi={10.1007/s00209-026-04023-6},
}
\bib{hlk}{article}{
  author={Hausen, J\"urgen}, author={Keicher, Simon}, author={Laface, Antonio},
  title={Computing Cox rings}, journal={Math. Comp.}, volume={85},
  number={297}, date={2016}, pages={467--502},
}
\bib{hau}{article}{
  author={Hausen, J\"urgen}, title={Cox rings and combinatorics. II},
  journal={Mosc. Math. J.}, volume={8}, number={4}, date={2008}, pages={711--757},
}
\bib{kl}{article}{
  author={Kleinschmidt, Peter}, title={A classification of toric varieties with few generators},
  journal={Aequationes Math.}, volume={35}, number={2--3}, date={1988}, pages={254--266},
}
\bib{ko}{article}{
  author={Kondo, Daiki}, author={Ogata, Shoetsu},
  title={On multiplication maps of ample bundles with nef bundles on toric surfaces},
  journal={Interdiscip. Inform. Sci.}, volume={14}, number={2},
  date={2008}, pages={183--190}, doi={10.4036/iis.2008.183},
}
\bib{ott}{article}{
  author={Ottem, John Christian},
  title={Birational geometry of hypersurfaces in products of projective spaces},
  journal={Math. Z.}, volume={280}, number={1--2}, date={2015}, pages={135--148},
}
\bib{rs}{article}{
  author={Ravindra, G. V.}, author={Srinivas, V.},
  title={The Grothendieck--Lefschetz theorem for normal projective varieties},
  journal={J. Algebraic Geom.}, volume={15}, number={3}, date={2006}, pages={563--590},
}
\bib{rs2}{article}{
  author={Ravindra, G. V.}, author={Srinivas, V.},
  title={The Noether--Lefschetz theorem for the divisor class group},
  journal={J. Algebra}, volume={322}, number={9}, date={2009}, pages={3373--3391},
}
\bib{reid}{misc}{
  author={Reid, Miles}, title={Graded rings and birational geometry},
  date={2000}, url={https://mreid.warwick.ac.uk/3folds/Ki/Ki.pdf},
}
\bib{stacks-project}{misc}{
  author={{The Stacks Project Authors}}, title={The Stacks Project},
  date={2026}, url={https://stacks.math.columbia.edu},
}
\end{biblist}
\end{bibdiv}
\end{document}